\documentclass[reqno]{amsart}

\usepackage{hyperref}
\usepackage{bbm}
\usepackage{amsfonts,amsmath,amsxtra,amsthm,amssymb,latexsym,enumerate}
\usepackage{color}

\newcommand*{\mailto}[1]{\href{mailto:#1}{\nolinkurl{#1}}}
\newcommand{\arxiv}[1]{\href{http://arxiv.org/abs/#1}{arXiv:#1}}
\newcommand{\msc}[1]{\href{http://www.ams.org/msc/msc2020.html?t=&s=#1}{#1}}

\newcommand{\ack}{\section*{Acknowledgments}}

\newtheorem{theorem}{Theorem}[section]
\newtheorem{corollary}[theorem]{Corollary}

\newtheorem{lemma}[theorem]{Lemma}
\newtheorem{hypothesis}{Hypothesis}[section]
\theoremstyle{definition}
\newtheorem{definition}[theorem]{Definition}
\newtheorem{example}[theorem]{Example}
\newtheorem{remark}[theorem]{Remark}

\newcommand{\wh}{\widehat}
\newcommand{\id}{{\mathbbm 1}}
\newcommand{\intrL}{\eta}

\numberwithin{equation}{section}

\DeclareMathOperator{\dom}{dom}

\DeclareMathOperator{\loc}{loc}

\def\mG{\mathsf{G}} 
  
\newcommand\R{{\mathbb{R}}}
\newcommand\C{{\mathbb{C}}}
\newcommand\Z{{\mathbb{Z}}}

\newcommand{\gq}{\mathfrak{q}}
\newcommand{\gQ}{\mathfrak{Q}}

\newcommand\Ei{e_{\imath}}
\newcommand\Et{e_{\tau}}

\newcommand\Deg{{\rm{Deg}}}

\newcommand\cA{{\mathcal{A}}}
\newcommand\cC{{\mathcal{C}}}
\newcommand\cD{{\mathcal{D}}}

\newcommand\cI{{\mathcal{I}}}

\newcommand\cG{{\mathcal{G}}}
\newcommand\cE{{\mathcal{E}}}
\newcommand\cM{{\mathcal{M}}}
\newcommand\cP{{\mathcal{P}}}
\newcommand\cV{{\mathcal{V}}}

\newcommand\bH{{\mathbf{H}}}
\newcommand\rh{{\mathbf{h}}}
\newcommand\rH{{\rm{H}}}

\newcommand\rI{{\rm{I}}}
\newcommand\E{{\rm{e}}}
\newcommand\vol{{\rm{vol}}}

\newcommand\I{{\rm{i}}}
\newcommand\rD{{\rm{d}}}

\newcommand\f{{\bf{f}}}

\begin{document}

\title[Laplacians on Infinite Graphs]{Laplacians on Infinite Graphs:\\ discrete vs continuous}

\author[A. Kostenko]{Aleksey Kostenko}
\address{Faculty of Mathematics and Physics\\ University of Ljubljana\\ Jadranska ul.\ 19\\ 1000 Ljubljana\\ Slovenia\\ and 
Faculty of Mathematics\\ University of Vienna\\
Oskar-Morgenstern-Platz 1\\ 1090 Vienna\\ Austria}
\email{\mailto{Aleksey.Kostenko@fmf.uni-lj.si}}

\author[N. Nicolussi]{Noema Nicolussi}
\address{Faculty of Mathematics\\ University of Vienna\\ 
Oskar-Morgenstern-Platz 1\\ 1090 Vienna\\ Austria}
\email{\mailto{noema.nicolussi@univie.ac.at}}

\thanks{{\it Research supported by the Austrian Science Fund (FWF) 
under Grants No.~P~28807, I~4600~(A.K.) and J4497~(N.N.), and by the Slovenian Research Agency (ARRS) under Grant No.~N1-0137~(A.K.)}}
\thanks{This is an extended version of the invited lecture of one of us (A.K.) at the 8th European Congress of Mathematics in Portoro\v{z} on 22 June 2021; the recording is available at this link \url{https://www.youtube.com/watch?v=NnlHVpee-CE}}

\keywords{Graph, metric graph, Laplacian, intrinsic metric, random walks}
\subjclass[2010]{Primary \msc{34B45}; Secondary \msc{47B25}; \msc{81Q10}}

\begin{abstract}
There are two main notions of a Laplacian operator associated with
graphs: discrete graph Laplacians and continuous Laplacians on metric graphs
(widely known as quantum graphs). Both objects have a venerable history as
they are related to several diverse branches of mathematics and mathematical
physics. 
The existing literature usually treats these two Laplacian operators separately. In this overview, we will focus on the relationship between them (spectral and parabolic properties). Our main conceptual message
is that these two settings should be regarded as complementary (rather than
opposite) and exactly their interplay leads to important further insight on
both sides.
\end{abstract}

\maketitle


\section{Introduction}

Laplacian operators on graphs have a long history and enjoy deep connections to several branches of mathematics and mathematical physics. There are two different notions of Laplacians appearing in this context: the key features of (continuous) {\em Laplacians on metric graphs}, which are also known as {\em quantum graphs}, include their use as simplified models of complicated quantum systems (see, e.g., \cite{bk13},
\cite{ekkst08}, \cite{exko}, \cite{post}) and the appearance of metric graphs in tropical and algebraic geometry, where they serve as non-Archimedean analogues of Riemann surfaces (see, e.g., \cite{baru10}, \cite{drva}).
The subject of {\em discrete Laplacians on graphs} is even wider and a partial overview of the immense literature can be found in \cite{bar}, \cite{chung}, \cite{cdv}, \cite{klwBook}, \cite{woe}.

The study of both types of graph Laplacians is heavily influenced by the corresponding investigations in the manifold setting (e.g., spectral geometry of manifolds). In fact, one can also put Laplacians on manifolds, metric and discrete graphs under the overarching umbrella of {\em Dirichlet forms}, which provides the systematic framework for studying heat and diffusion processes. From this perspective, metric graph Laplacians have much in common with Laplacians on manifolds since both can be treated in the framework of strongly local Dirichlet forms. Moreover, the notion of an intrinsic metric, first mentioned by E.~B.~Davies and later emphasized by M.~Biroli, U.~Mosco, and K.-T.~Sturm (see, e.g., \cite{stu}), allows to directly transfer many important results from manifolds to the abstract setting of strongly local Dirichlet forms (and hence metric graph Laplacians).

In contrast to this, discrete graph Laplacians are difference operators and hence provide examples of non-local operators (e.g., no Leibniz rule). In particular, difficulties in analyzing random walks on graphs often stem from exactly this fact. On the other hand, this area has seen a tremendous progress in the last decade. In our opinion, the successful introduction and systematic use of the notion of {\em intrinsic metrics on graphs} played (and continues to play!) a major role in this breakthrough (see the very recent monograph \cite{klwBook}).

Despite a vast interest in both types of graph Laplacians, the existing literature usually treats them separately. 
In the present overview, we mainly focus on the relationship between them and survey connections on different levels (spectral and parabolic properties). This leads to a systematic way of connecting the settings and several applications. Our main conceptual message is that discrete and continuous graph Laplacians should be regarded as complementary (rather than opposite) and exactly their interplay leads to important further insight on both sides. This relationship can also be formulated in the language of intrinsic metrics. Indeed, a large class of intrinsic metrics on discrete graphs is obtained as restrictions to vertices of intrinsic metrics on (weighted) metric graphs. In particular, from this perspective metric graphs indeed serve as bridge between graphs and manifolds, a heuristic principle which is often mentioned in context with graph Laplacians.
Let us also mention that the stochastic side of these connections, namely the approach of using Brownian motion on metric graphs to study random walks on discrete graphs, has been employed several times in the literature \cite{baba03}, \cite{fo13}, \cite{fo}, \cite{hua}, \cite{hush14}, \cite{var85b} (see also references therein).

Most of the results presented here are carefully explained in the recent monograph \cite{kn21}, which also contains many other results not mentioned in this text.

\noindent
\ack Some of the results reviewed here were obtained in collaboration with Pavel Exner, Mark Malamud, and Delio Mugnolo and it is our great pleasure to thank them. We are also grateful to Omid Amini, Matthias Keller and Wolfgang Woess for numerous useful and stimulating discussions.

\section{Preliminaries} \label{sec:prelim}

\subsection{Graphs}\label{ss:II.01}

Let us first recall basic notions (we mainly follow the terminology in \cite{die}). Let $\cG_d = (\cV,\cE)$ be an undirected {\em graph}, that is, $\cV$ is a finite or countably infinite set of vertices and $\cE$ is a finite or countably infinite set of edges. 
Two vertices $u$, $v\in \mathcal{V}$ are called {\em neighbors} and we shall write $u\sim v$ if there is an edge $e_{u,v}\in \mathcal{E}$ connecting $u$ and $v$. For every $v\in \mathcal{V}$, we define $\cE_v$ as the set of edges incident to $v$. We stress that we allow {\em multigraphs}, that is, we allow {\em multiple edges} (two vertices can be joined by several edges) and {\em loops} (edges from one vertex to itself). Graphs without loops and multiple edges are called {\em simple}. 

\begin{example}[Cayley graphs]\label{ex:Cayley}
Let $\mG$ be a finitely generated group and let $S$ be a generating set of $\mG$. We shall always assume that 
\begin{itemize}
\item $S$ is symmetric, $S=S^{-1}$ and finite, $\#S<\infty$,
\item the identity element of $\mG$ does not belong to $S$ (this excludes loops).
\end{itemize} 
The {\em Cayley graph} $\cG_C = \cC(\mG,S)$ of $\mG$ with respect to $S$ is the simple graph whose vertex set coincides with $\mG$ and two vertices $x,y\in\cG$ are neighbors if and only if $xy^{-1}\in S$. 
\end{example}

Sometimes it is convenient to assign an {\em orientation} on $\cG_d$: to each edge $e\in\cE$ one assigns the pair $(\Ei,\Et)$ of its {\em initial} $\Ei$ and {\em terminal} $\Et$ vertices. We shall denote the corresponding oriented graph by $\vec{\cG}_d = (\cV,\vec{\cE})$, where $\vec{\cE}$ denotes the set of oriented edges. Notice that for an oriented loop we do distinguish between its initial and terminal vertices. Next, for every vertex $v\in\cV$, set 
\begin{align}\label{eq:Ev_pm}
{\cE}^+_v & = \big\{(\Ei,\Et) \in \vec{\cE}\,|\, \Ei = v\big\}, & {\cE}_v^- & = \big\{(\Ei,\Et) \in \vec{\cE} \,|\,  \Et = v\big\},
\end{align}
and let $\vec{\cE}_v$ be the disjoint union of outgoing $\cE_v^+$ and incoming $\cE_v^-$ edges,
\begin{align}\label{eq:vecEv}
\vec{\cE}_v & := {\cE}_v^+ \sqcup {\cE}_v^- = \vec{\cE}_v^+ \cup \vec{\cE}_v^-, &  \vec{\cE}_v^\pm & := \big\{(\pm,e)\,|\, e\in \cE_v^\pm\big\}.
\end{align}
The {\em (combinatorial) degree} of $v\in\cV$ is 
\begin{align}\label{eq:combdeg}
\deg(v):=  \#(\vec{\cE}_v ) = \#(\vec{\cE}_v^+ ) + \#(\vec{\cE}_v^- )  = \#(\cE_v ) + \#\{e\in\cE_v|\, e\ \text{is a loop}\}.
\end{align}
Notice that if $\cE_v$ has no loops, then $\deg(v) = \#(\cE_v)$. The graph $\cG_d$ is called {\em locally finite}  if $\deg(v)<\infty$ for all $v\in\cV$. 

 A sequence of (unoriented) edges $\cP = (e_{v_0, v_1}, e_{v_1, v_2}, \dots, e_{v_{n-1}, v_n})$, where $e_{v_i, v_{i+1}}$ connects the vertices $v_i$ and $v_{i+1}$, is called a {\em path} of (combinatorial) length $n\in \Z_{\ge 0}\cup \{\infty\}$.  Notice that for simple graphs each path $\cP$ can be identified with its sequence of vertices $\cP = (v_k)_{k=0}^n$.
 A graph $\cG_d$ is called {\em connected} if for any  two vertices there is a path connecting them. 
 
We shall always make the following assumptions on the geometry of $\cG_d$:

\begin{hypothesis}\label{hyp:graph01}
$\cG_d$ is connected and locally finite. 
\end{hypothesis}

\begin{remark}
We assume connectivity for convenience reasons only (one can always consider each connected component of a graph separately). However, the assumption that a graph is locally finite is indeed important in our considerations. 
\end{remark}

\subsection{Metric graphs}\label{ss:II.02}

Assigning each edge $e\in\cE$ a finite length $|e| \in (0,\infty)$, we can naturally associate with $(\cG_d,|\cdot|) = (\cV,\cE,|\cdot|)$ a metric space $\cG$: first, we identify each edge $e \in \cE$ with the copy of the interval $\cI_e = [0, |e|]$, which also assigns an orientation on $\cE$ upon identification of $\Ei$ and $\Et$ with the left, respectively, right endpoint of $\cI_e$. The topological space $\cG$ is then obtained by ``glueing together" the ends of edges corresponding to the same vertex $v$ (in the sense of a topological quotient, see, e.g., \cite[Chap.~3.2.2]{bbi}). 
The topology on $\cG$ is metrizable by the {\em length metric} $\varrho_0$ --- the distance between two points $x,y \in\cG$ is defined as the arc length of the ``shortest path" connecting them (such a path does not necessarily exist and one needs to take the infimum over all paths connecting $x$ and $y$). 

A \emph{metric graph} is a (locally compact) metric space $\cG$ arising from the above construction for some collection $(\cG_d, |\cdot|) =(\cV, \cE, |\cdot|)$. More specifically, $\cG$ is then called the \emph{metric realization} of $(\cG_d, |\cdot|)$, and a pair $(\cG_d, |\cdot|)$ whose metric realization coincides with $\cG$ is called a \emph{model} of $\cG$. For a thorough discussion of metric graphs as topological and metric spaces we refer to \cite[Chap.~I]{hae}. 

\begin{remark}\label{rem:II.mr=lengthspace}
Let us stress that a metric graph $\cG$ equipped with the length metric $\varrho_0$ (or with any other path metric) is a {\em length space} (see \cite[Chap.~2.1]{bbi} for definitions and further details). Notice also that complete, locally compact length spaces are {\em geodesic}, that is, every two points can be connected by a shortest path.
\end{remark}

Clearly, different models may give rise to the same metric graph. Moreover,  any metric graph has infinitely many models (e.g., they can be constructed by subdividing edges using vertices of degree two). 
A model $(\cV,\cE, |\cdot|)$ is called {\em simple} if the corresponding graph $(\cV,\cE)$ is simple. In particular, every locally finite metric graph has a simple model and this indicates that restricting to simple graphs, that is, assuming in addition to Hypothesis~\ref{hyp:graph01} that $\cG_d$ has no loops or multiple edges, would not be a restriction at all when dealing with metric graphs. 

\begin{remark}\label{rem:Models}
In most parts of our paper, we will consider a metric graph together with a fixed choice of its model. In this situation, we will usually be slightly imprecise and do not distinguish between these two objects. In particular, we will denote both objects by the same letter $\cG$ and write $\cG = (\cV, \cE, |\cdot|)$ or $\cG= (\cG_d,|\cdot|)$. 
\end{remark}

\begin{remark}[Metric graph as a 1d manifold with singularities]\label{rem:MGasM1}
Sometimes it is useful to consider metric graphs as one-dimensional manifolds with singularities.
Since every point $x\in\cG$ has a neighborhood isomorphic to a star shaped set
\begin{align} \label{eq:star}
\cE(\deg(x),r_x) := \big\{z= r\E^{2\pi \I k/\deg(x)}|\, r\in [0,r_x),\ k=1,\dots,\deg(x) \big\}\subset \C,
\end{align}
 one may introduce the set of {\em tangential directions} $T_x(\cG) $ at $x$ as the set of unit vectors $\E^{2\pi \I k/\deg(x)}$, $k=1,\dots, \deg(x)$. 
Then all vertices $v\in\cV$ with $\deg(v)\ge 3$ are considered as {\em branching points/singularities}
and vertices $v\in\cV$ with $\deg(v)=1$ as {\em boundary points}. 
Notice that for every vertex $v\in\cV$ the set of tangential directions $T_v(\cG)$ can be identified with $\vec{\cE}_v$. 
 If there are no loop edges at the vertex $v \in \cV$, then $T_v(\cG)$ is 
identified with $\cE_v$ in this way. 
\end{remark}

\section{Graph Laplacians} \label{ss:III.01}

When speaking about graph Laplacians, usually one of the operators considered in the next two examples is meant.

\begin{example}[Combinatorial Laplacian]\label{ex:LaplComb} 
For a simple graph $\cG_d = (\cV,\cE)$ satisfying Hypothesis~\ref{hyp:graph01}, the so-called {\em combinatorial Laplacian} is defined on $C(\cV)$ by 
\begin{align}\label{eq:LaplComb}
(L_{\rm comb} f)(v) = \sum_{u\sim v} f(v) - f(u) = \deg(v) f(v) - \sum_{u\sim v} f(u),\qquad v\in \cV. 
\end{align}
Here $C(\cV)$ is the set of complex-valued functions on a countable set $\cV$. 
Notice that the second summand on the RHS,
\begin{align*} 
(\cA f)(v) =  \sum_{u\sim v} f(u),\qquad v\in \cV, 
\end{align*}
is nothing but the operator generated by the adjacency matrix of $\cG_d$, which explains the name of $L_{\rm comb}$. The combinatorial Laplacian plays a crucial role in many areas of mathematics, physics, and engineering. In particular, the relationship between the spectral properties of $L_{\rm comb}$ and various graph parameters is one of the core topics within the field of {\em Spectral Graph Theory} (see \cite{chung}, \cite{cdv} for further details).
\end{example} 

\begin{example}[Normalized Laplacian]\label{ex:LaplNorm} 
Assuming again that $\cG_d = (\cV,\cE)$ is a simple graph satisfying Hypothesis~\ref{hyp:graph01}, consider another operator defined by 
\begin{align}\label{eq:LaplNorm}
(L_{\rm norm} f)(v) = \frac{1}{\deg(v)}\sum_{u\sim v} f(v) - f(u) = f(v) - \frac{1}{\deg(v)}\sum_{u\sim v} f(u)
\end{align}
for every $v\in \cV$. 
The second summand on the RHS,
\begin{align*} 
(\cM f)(v) =  \frac{1}{\deg(v)}\sum_{u\sim v} f(u), 
\end{align*}
is the so-called {\em Markov (averaging) operator}. Notice that due to our assumptions on $\cG_d$, $\cM$ is a stochastic matrix known as the {\em transition matrix} for the simple random walk on the graph. 
The normalized Laplacian serves as a generator of a simple random walk on $\cG_d$ (see, e.g., \cite{bar}, \cite{woe}). 
\end{example} 

\begin{remark}
If the underlying graph $\cG_d$ is {\em regular} ($\deg$ is constant on $\cV$; for instance, Cayley graphs are regular), then $L_{\rm comb} = \deg \cdot L_{\rm norm} = \deg\cdot \rI - \cA$. However, in general these two Laplacians may have very different properties. For instance, $L_{\rm norm}$ generates a bounded operator in $\ell^2(\cV;\deg)$  and $L_{\rm comb}$ gives rise to a bounded operator in $\ell^2(\cV)$ only if $\cG_d$ has bounded geometry (see Remark~\ref{rem:bddLapl} below).
\end{remark}

The above two examples can be put into a much more general framework. Namely, let $\cV$ be a countable set. A function $m\colon \cV\to (0,\infty)$ defines a measure of full support on $\cV$ in an obvious way. A pair $(\cV,m)$ is called a {\em discrete measure space}. The set of square summable (w.r.t. $m$) functions 
\[
\ell^2(\cV;m) = \Big\{ f\in C(\cV)\,|\,\, \|f\|^2_{\ell^2(\cV;m)}:= \sum_{v\in\cV} |f(v)|^2 m(v) <\infty \Big\}
\]
has a natural Hilbert space structure. 

Suppose $b\colon \cV\times\cV \to [0,\infty)$ satisfies the following conditions:
\begin{itemize}
\item[(i)] {\em symmetry}: $b(u,v) = b(v,u)$ for each pair $(u,v)\in \cV\times\cV$,
\item[(ii)] {\em vanishing diagonal}: $b(v,v) = 0$ for all $v\in\cV$,
\item[(iii)] {\em locally finite}:   $\#\{ u\in\cV\,|\, b(u,v)\neq 0\} < \infty$  for all $v\in\cV$\footnote{In fact, using the form approach one can considerably relax this condition by replacing it  with the {\em local summability}: $\sum_{v\in\cV} b(u,v) <\infty$ for all $u\in\cV$.}.
\item[(iv)] {\em connected}: for any $u,v\in \cV$ there is a finite collection $(v_k)_{k=0}^n \subset \cV$ such that $u=v_0$, $v=v_n$ and $b(v_{k-1},v_k)>0$ for all $k\in \{1,\dots,n\}$.
\end{itemize}
Following \cite{kl12}, \cite{klwBook}, $b$ is called a {\em (weighted) graph} over $\cV$ or over $(\cV,m)$ if in addition  a measure $m$ of full support on $\cV$ is given ($b$ is also called an {\em edge weight}).  To simplify notation, we shall denote a graph $b$  over $(\cV,m)$ by $(\cV,m;b)$. 

\begin{remark}\label{rem:simplevsmult}
To any graph $b$ over $\cV$, we can naturally associate a simple combinatorial graph $\cG_b$. Namely, the vertex set of $\cG_b$ is $\cV$ and its edge set $\cE_b$ is defined by calling two vertices $u,v\in\cV$ neighbors, $u\sim v$, exactly when $b(u,v)>0$. Clearly, $\cG_b = (\cV,\cE_b)$ is an undirected graph in the sense of Section~\ref{ss:II.01}. Let us stress, however, that the constructed {\em graph $\cG_b$ is always simple}. 
\end{remark}

The {\em (formal) Laplacian} $L = L_{m,b}$ associated to a graph $b$ over $(\cV,m)$ is given by
\begin{align}\label{eq:LaplDiscr}
(L f)(v) = \frac{1}{m(v)}\sum_{u\in\cV} b(v,u) (f(v) - f(u)).
\end{align} 
It acts on functions $f\in C(\cV)$ and this naturally leads to the {\em maximal} Laplacian $\rh$ in $\ell^2(\cV;m)$ defined by 
\begin{align}\label{eq:LaplDiscrMax}
\rh & = L\upharpoonright\dom(\rh), & \dom(\rh) & = \{f\in \ell^2(\cV;m)\,|\,  Lf\in\ell^2(\cV;m)\}.
\end{align}
This operator is closed, however, if $\cV$ is infinite, it is not symmetric in general (cf. \cite[Theorem~6]{kl12}). 
Taking into account that $b$ is locally finite, it is straightforward to verify that $C_c(\cV)\subseteq \dom(\rh)$. Therefore, we can introduce the {\em minimal} Laplacian $\rh^0$ as the closure in $\ell^2(\cV;m)$ of the {\em pre-minimal} Laplacian
\begin{align}\label{eq:pmLaplDiscr}
\rh' = L\upharpoonright\dom(\rh'),\qquad \dom(\rh')= C_c(\cV).
\end{align}
Then $\rh'\subseteq \rh^0\subseteq \rh$ and $(\rh')^\ast = (\rh^0)^\ast  = \rh$.  If $\rh^0  = \rh$, then $\rh$ is {\em self-adjoint} as an operator in the Hilbert space $\ell^2(\cV; m)$ (and $\rh'$ is called {\em essentially self-adjoint}). The problem of self-adjointness is a classical topic, which is of central importance in quantum mechanics (see, e.g., \cite[Chap.~VIII.11]{RSI}). We shall return to this issue in Section~\ref{ss:applSA}. Let us now only mention that the self-adjointness takes place whenever $L = L_{m,b}$ gives rise to a bounded operator on $\ell^2(\cV;m)$.
It is rather well known (see, e.g., \cite[Lemma~1]{dav}, \cite[Theorem~11]{kl10}, \cite[Rem.~1]{susy}) that the Laplacian $L = L_{m,b}$ is bounded on $\ell^2(\cV;m)$ if and only if the weighted degree function
  $\Deg\colon \cV \to [0,\infty)$ given by 
\begin{align}\label{eq:WDegDef}
\Deg\colon v\mapsto \frac{1}{m(v)} \sum_{u\in \cV} b(u,v)
\end{align}
is bounded on $\cV$. In this case $\rh^0 = \rh$ and $\|\Deg\|_{\infty}\le \|\rh\|_{\ell^2(\cV;m)} \le 2\|\Deg\|_{\infty}$. 

\begin{remark}\label{rem:bddLapl}
For the combinatorial Laplacian $L_{\rm comb}$, we have $\Deg_{\rm comb}(v) = \deg(v)$ and hence $L_{\rm comb}$ is bounded exactly when $\cG_d$ has bounded geometry. For the normalized Laplacian $L_{\rm norm}$, $\Deg_{\rm norm}(v) \le 1$ for all $v\in\cV$ and hence $\|L_{\rm norm}\| \le 2$.
\end{remark}

There is another way to associate a self-adjoint operator with $L$ in $\ell^2(\cV;m)$. With each graph $b$ one can associate the {\em energy form} $\gq \colon C(\cV)\to [0,\infty]$ defined by 
\begin{align}\label{eq:gqLaplDiscr}
\gq [f] = \gq_{b}[f]:= \frac{1}{2}\sum_{u,v\in\cV} b(v,u) |f(v) - f(u)|^2.
\end{align} 
Functions $f\in C(\cV)$ such that $\gq[f] <\infty$ are called {\em finite 
energy functions}. Clearly\footnote{Actually, it suffices to assume that $b$ satisfies the local summability condition, see \cite{kl12},\cite{klwBook}.}, $C_c(\cV)$ belongs to the set $\cD(\gq)$ of finite energy functions and $\langle \rh f,f\rangle_{\ell^2(m)} = \gq[f]$ for all $f\in C_c(\cV)$. If $b$ is a graph over $(\cV,m)$, introduce the graph norm  
\begin{align}
\| f\|^2_{\gq} := \gq[f] + \|f\|^2_{\ell^2(\cV;m)}
\end{align}
for all $f\in \cD(\gq) \cap \ell^2(\cV;m)= : \dom(\gq)$. Clearly, $\dom(\gq)$ 
is the maximal domain of definition of the form $\gq$ in the Hilbert space $\ell^2(\cV;m)$; let us denote this form by $\gq_N$. Restricting further to compactly supported functions and then taking the graph norm closure, we get another form:
\[
\gq_D:= \gq\upharpoonright \dom(\gq_D), \qquad \dom(\gq_D) := \overline{C_c(\cV)}^{\|\cdot\|_{\gq}}.
\]
It turns out that both $\gq_D$ and $\gq_N$ are {\em Dirichlet forms} (for 
definitions see \cite{fuk10}). Moreover, $\gq_D$ is a  {\em regular Dirichlet form}. It turns out that the converse is also true: {\em ``every (irreducible) regular Dirichlet form over $(\cV,m)$ arises as the 
energy form $\gq_D$ for some (connected) graph $b$ over $(\cV,m)$"} (this claim is wrong as stated, however, to make it correct one needs to replace locally finite by the local summability condition on $b$ and also to allow killing terms, see \cite[Theorem~7]{kl12}). 

\begin{remark} 
The notion of {\em irreducibility} for Dirichlet forms on graphs is closely connected with the notion of {\em connectivity}. Recall that a graph $b$ is called {\em connected} if the corresponding graph $\cG_b$ is connected. Then the regular Dirichlet form $\gq_D$ is irreducible exactly when the underlying graph $b$ is connected (e.g., \cite[Chap.~1.4]{klwBook}).
\end{remark}

Now using the representation theorems for quadratic forms (see, e.g., \cite{kato}) one can associate in $\ell^2(\cV;m)$ the self-adjoint operators 
$\rh_D$ and $\rh_N$, the so-called {\em Dirichlet}\label{not:graphDirLapl} and {\em Neumann Laplacians}\label{not:graphNeuLapl} over $(\cV,m)$, with, respectively, $\gq_D$ and $\gq_N$. Usually, it is a rather nontrivial 
task to provide an explicit description of the operators $\rh_D$ and, especially, $\rh_N$\footnote{In fact, to decide whether $\rh_N$ and $\rh_D$ coincide for given $b$ and $m$, or equivalently that $\gq_N = \gq_D$, is already a highly nontrivial problem. This property is related to the uniqueness of a {\em Markovian extension}. For further details we refer to \cite{klwBook}, \cite{kmn19}, \cite[Chap.~7.2]{kn21}.}. However, the following abstract description always holds,
\begin{align}\label{eq:LaplDiscrDir}
\rh_D = \rh\upharpoonright\dom(\rh_{D}),\qquad \dom(\rh_D) = \dom(\rh)\cap \dom(\gq_D),
\end{align}
which also implies that $\rh_D$ is the {\em Friedrichs extension} of the adjoint $\rh^0 = \rh^\ast$ to $\rh$.

\section{Laplacians on metric graphs}\label{sec:LaplMetric}

\subsection{Function spaces on metric graphs}\label{ss:IV.01}  
Let $\cG$ be a metric graph together with a fixed model $(\cV,\cE,|\cdot|)$. Let also $\mu\colon \cE\to (0,\infty)$ be a weight function assigning 
a positive weight $\mu(e)$ to each edge $e\in\cE$. We shall assume that edge weights are orientation independent and we set 
$\mu(\vec{e}) = \mu(e)$ for all $\vec{e}\in \vec{\cE}_v$, $v\in\cV$.  
Identifying every edge $e\in\cE$ with a copy of $\cI_e = [0,|e|]$, we can introduce Lebesgue and Sobolev spaces on edges and also on $\cG$.  First of all, with the weight $\mu$ we associate the measure $\mu$ on $\cG$ defined as the edgewise scaled Lebesgue measure such that $\mu(\rD x) = 
\mu(e)\rD x_e$ on every edge $e\in\cE$. Thus, we can define the Hilbert space $L^2(\cG;\mu)$ of measurable functions $f\colon \cG\to \C$ which are 
square integrable w.r.t. the measure $\mu$ on $\cG$. Similarly, one defines the Banach spaces $L^p(\cG;\mu)$ for $p\in [1,\infty]$. In fact, if $p\in [1,\infty)$, then 
\[
L^p(\cG;\mu) \cong  \Big\{f=(f_e)_{e\in\cE}\big|\, f_e\in L^p(e;\mu),\ \sum_{e\in\cE}\|f_e\|^p_{L^p(e;\mu)}<\infty\Big\},
\]
where 
\[
\|f_e\|^p_{L^p(e;\mu)} = \int_{e}|f_e(x_e)|^p\mu(\rD x_e) =  \mu(e)\int_{e}|f_e(x_e)|^p\,\rD x_e.
\]
 If $\mu(e) = 1$, then we shall simply write $L^p(e)$. 
Next, the subspace of compactly supported $L^p$ functions will be denoted 
by $L^p_c(\cG;\mu)$. 
The space  $L^p_{\loc}(\cG;\mu)$  of locally $L^p$ functions consists of all measurable functions $f$ such that $fg\in L^p_c(\cG;\mu)$ for all $g\in C_c(\cG)$. Notice that both $L^p_{\loc}$ and $L^p_c$ are independent of the weight $\mu$.

For edgewise locally absolutely continuous functions on $\cG$, let us denote by $\nabla$ the edgewise first derivative,
\begin{align}\label{eq:nabla}
\nabla\colon f\mapsto f'.
\end{align}
Then for every edge $e\in\cE$, 
 \begin{align*}
H^1(e) & = \{f\in AC(e)\,|\, \nabla f\in L^2(e)\}, & H^2(e) & = \{f\in H^1(e)\,|\,  \nabla f\in H^1(e)\},
 \end{align*}
are the usual Sobolev spaces (upon the identification of $e$ with $\cI_e = [0,|e|]$), and $AC(e)$ is the space of absolutely continuous functions on $e$. 
Let us denote by $H^1_{\loc}(\cG\setminus\cV)$ and $H^2_{\loc}(\cG\setminus\cV)$ the spaces of measurable functions $f$ on $\cG$ such that their  edgewise restrictions belong to $H^1$, respectively, $H^2$, that is, 
\begin{align*}
H^j_{\loc}(\cG\setminus\cV) = \{f\in L^2_{\loc}(\cG)\, |\, f|_e\in H^j(e)\ \text{for all}\ e\in\cE\}
\end{align*}
for $j\in\{1,2\}$. Clearly, for each measurable $f \in H^2_{\loc}(\cG\setminus\cV)$ the following quantities
 \begin{align}\label{eq:tr_fe}
 f(\Ei) & := \lim_{x_e \to \Ei} f (x_e), & f (\Et) & := \lim_{x_e \to \Et} f (x_e),
 \end{align}
 and the normal derivatives 
 \begin{align}\label{eq:tr_fe'}
 \partial f (\Ei) & := \lim_{x_e \to \Ei} \frac{f (x_e) - f(\Ei)}{|x_e - \Ei|}, & 
 \partial f (\Et) & := \lim_{x_e \to \Et} \frac{f (x_e) - f(\Et)}{|x_e - \Et|},
 \end{align}
are well defined for all $e\in\cE$. We also need the following notation
\begin{align}\label{eq:DOMtr_fe}
f_{\vec{e}}(v) & := \begin{cases} f(\Ei), & \vec{e}\in \vec{\cE}_v^+, \\ f(\Et), & \vec{e} \in \vec{\cE}_v^-, \end{cases}
& \partial_{\vec{e}} f(v) & := \begin{cases} \partial f (\Ei), & \vec{e}\in \vec{\cE}_v^+, \\ \partial f (\Et), & \vec{e}\in \vec{\cE}_v^-, \end{cases}
\end{align}
for every $v\in\cV$ and $\vec{e}\in\vec{\cE}_v$. In the case of a loopless graph, the above notation simplifies since we can identify $\vec{\cE}_v$ with $\cE_v$ for all $v\in\cV$.  

\subsection{Kirchhoff Laplacians}\label{ss:IV.02}

Again, let $\cG$ be a metric graph together with a fixed model $(\cV,\cE,|\cdot|)$. 
Let  $\mu,\ \nu\colon \cE\to (0,\infty)$ be two edge weights on $\cG$ (for a given model). 
For every $e\in\cE$ consider the maximal operator $\rH_{e,\max}$ defined in $L^2(e;\mu)$ by
\begin{align}\label{eq:Hemax}
 \rH_{e,\max}f & = \tau_e f,\qquad \tau_e = -\frac{1}{\mu(x_e)}\frac{\rD}{\rD x_e}\nu(x_e)\frac{\rD}{\rD x_e}, \\
 \dom(\rH_{e,\max}) & = \big\{f\in L^2(e;\mu)\,|\, f,\ \nu f'\in AC(e),\ \tau_e f\in  L^2(e;\mu) \big\}.
\end{align}
Since $\mu$, $\nu$ are constant on $e$, $\dom(\rH_{e,\max})$ coincides with the Sobolev space $H^2(e)$. The maximal operator on $\cG$ is then defined in $L^2(\cG;\mu)$ as
\begin{align}\label{eq:Hmax}
\bH_{\max}  = \bigoplus_{e\in \cE} \rH_{e,\max}. 
\end{align}
Clearly, for each $f \in \dom(\bH_{\max})$ the quantities \eqref{eq:tr_fe}, \eqref{eq:tr_fe'}, and hence \eqref{eq:DOMtr_fe}
are well defined for all $e\in\cE$. Now, in order to reflect the underlying graph structure, we impose at each 
vertex $v\in\cV$ the {\em Kirchhoff boundary conditions} 
\begin{align}\label{eq:kirchhoff}
\begin{cases} f\ \text{is continuous at}\ v,\\[1mm] 
\sum\limits_{\vec{e}\in \vec{\cE}_v} \nu(e)\partial_{\vec{e}} f(v) = 0. \end{cases} 
\end{align} 

To motivate our definition, consider $\nabla$ as the differentiation operator on $\cG$ acting on functions which are edgewise locally absolutely continuous and also continuous at the vertices.
Notice that when considering $\nabla$ as an operator acting from $L^2(\cG;\mu)$ to $L^2(\cG;\nu)$, its formal adjoint $\nabla^\dagger$ acting from $L^2(\cG;\nu)$ to $L^2(\cG;\mu)$ acts edgewise as
\begin{align}
\nabla^\dagger \colon f\mapsto -\frac{1}{\mu} (\nu f)'.
\end{align}
Thus, the weighted Laplacian $\Delta$ acting in $L^2(\cG;\mu)$, written in the divergence form 
$\Delta \colon f\mapsto  -\nabla^\dagger (\nabla f)$,
acts edgewise as the following divergence form Sturm--Liouville operator
\begin{align}\label{eq:LaplMetrG}
\Delta\colon f\mapsto \frac{1}{\mu}(\nu f')'.
\end{align}
The continuity assumption imposed on $f$ results for $\Delta$ in a one-parameter family of symmetric boundary conditions (the so-called $\delta$-coupling). In the present text, with the Laplacian $\Delta$ acting on $\cG$ we shall always associate  
the {\em Kirchhoff} vertex conditions~\eqref{eq:kirchhoff}. 
In particular, imposing these boundary conditions on the maximal domain yields the \emph{(maximal) Kirchhoff Laplacian}:
\begin{align}\label{eq:H}
\begin{split}
	\bH  =  -\Delta\upharpoonright &{\dom(\bH)},\\
	& \dom(\bH ) = \{f\in \dom(\bH_{\max})\,|\, f\ \text{satisfies}~\eqref{eq:kirchhoff}\ \text{on}\ \cV\}.
\end{split}
\end{align}
 
\subsection{Energy forms} \label{ss:IV.03}

Restricting further to compactly supported functions we end up with the pre-minimal operator
\begin{align}\label{eq:Halpha0}
	\bH'  =  -\Delta\upharpoonright {\dom(\bH')},\qquad 
	 \dom(\bH')  = \dom(\bH) \cap C_c(\cG).
\end{align}
Integrating by parts one obtains for all $f \in \dom(\bH')$
\begin{align}\label{eq:QFalpha}
	\langle\bH' f, f\rangle_{L^2} = \int_\cG |\nabla f(x)|^2 \, \nu(\rD x) =: \gQ[f] ,
\end{align}
which implies that $\bH'$ is a nonnegative symmetric operator in $L^2(\cG;\mu)$. We define $\bH^0$ as the closure of $\bH'$ in $L^2(\cG;\mu)$. It is standard to show that
\begin{align}\label{eq:H0*=HA}
		(\bH')^\ast = \bH.
\end{align}
In particular, the equality $\bH^0 = \bH$ holds if and only if $\bH$ is self-adjoint (or, equivalently, $\bH'$ is essentially self-adjoint). 

With the form $\gQ$ we associate two spaces: the Sobolev space $H^1(\cG) = H^1(\cG;\mu,\nu)$ is defined as the subspace of $L^2(\cG;\mu)$ consisting of continuous functions, which are edgewise absolutely continuous and have finite energy $\gQ[f] < \infty$. Equipping $H^1(\cG)$ with the standard graph norm turns it into a Hilbert space. Also, we define the space $H^1_0(\cG) = H^1_0(\cG;\mu,\nu)$ as the closure of compactly supported $H^1$ functions,
\[
H^1_0 = H^1_0(\cG;\mu,\nu):= \overline{H^1_c(\cG)}^{\|\cdot\|_{H^1(\cG;\mu,\nu)}},
\]
where $H^1_c(\cG) := H^1(\cG)\cap C_c(\cG)$. 
Restricting $\gQ$ to these spaces, we end up with two closed forms in $L^2(\cG;\mu)$:
\begin{align}
\gQ_D & = \gQ\upharpoonright{H^1_0}, & \gQ_N & = \gQ\upharpoonright{H^1}.
\end{align}
According to the representation theorem, they give rise to two self-adjoint nonnegative operators $\bH_D$ and $\bH_N$ in $L^2(\cG;\mu)$, the Dirichlet and Neumann Laplacians, respectively. Notice also that $\bH_D$ coincides with the Friedrichs extension of $\bH'$:
\[
\dom(\bH_D) = \dom(\bH)\cap H^1_0(\cG).
\]

\begin{remark}
Following the analogy with the Friedrichs extension, it might be tempting 
to think that the domain of the Neumann Laplacian $\bH_N$ is given by $\dom(\bH)\cap H^1(\cG)$. However, the operator defined on this domain has a 
different name --- the {\em Gaffney Laplacian} --- and it is not symmetric in general. Moreover, this operator is not always closed (see \cite{kn20}). 
\end{remark}

\section{Connections}\label{sec:connection}

One of the immediate ways to relate Laplacians on metric and discrete graphs is by noticing a connection between their harmonic functions. Despite being elementary, this observation lies at the core of many of our considerations and hence we briefly sketch it here. Every harmonic function $f$ on a weighted metric graph $(\cG, \mu, \nu)$ (i.e., $f$ satisfies $\Delta f = 0$), must be edgewise affine. The Kirchhoff conditions \eqref{eq:kirchhoff} imply that $f$ is continuous and, moreover, satisfies 
\[
\sum\limits_{\vec{e}\in \vec{\cE}_v} \nu(e)\partial_{\vec{e}} f(v) = 
\sum_{u\sim v} \sum_{\vec{e} \in \vec{\cE}_{u}\colon e\in\cE_v } \frac{\nu(e)}{|e|} \big (f(u)- f(v) \big) = 0
\]
at each vertex $v\in\cV$. 
This suggests to consider a discrete Laplacian \eqref{eq:LaplDiscr} with edge weights given by 
\begin{align}\label{eq:Bbndry}
b(u,v)  
=  \begin{cases}  
                 \sum_{\vec{e} \in \vec{\cE}_{u}\colon e\in\cE_v } \frac{\nu(e)}{|e|}, & u \neq v,  \\[1mm] 
                  \quad 0,  & u = v.
       \end{cases}
\end{align}
Indeed, then for every $\Delta$-harmonic function $f$ on the weighted metric graph $(\cG, \mu, \nu)$, its restriction to vertices $\f := f|_\cV$ is an $L$-harmonic function, that is, $L\f = 0$. Moreover, the converse is also true. Phrased in a more formal way, the  map
\begin{align}\label{eq:mapV}
\begin{array}{cccc}
\imath_\cV \colon & C(\cG) & \to & C(\cV) \\
 & f & \mapsto & f|_\cV 
 \end{array},
\end{align}
when restricted further to the space of continuous, edgewise affine functions 
on $\cG$ becomes bijective and establishes a bijective correspondence between $\Delta$-harmonic and $L$-harmonic functions. 
This indicates a possible connection between the corresponding Laplacians on $\cG$ and $\cG_d$ (this immediately connects, for instance, the corresponding Poisson and Martin boundaries). However, one also has to take into account the measures $\mu$ and $m$, that is, the vertex weight $m$ should be chosen in a way which connects the corresponding Hilbert spaces $L^2(\cG;\mu)$ and $\ell^2(\cV;m)$. The desired connection is given by the choice
\begin{align}\label{eq:Mbndry}
m\colon v\mapsto \sum_{\vec{e}\in\vec{\cE}_v} |e|\mu(e),\qquad v\in\cV,
\end{align}
under the additional assumption that $(\cG,\mu,\nu)$ has {\em finite intrinsic size}:
\begin{align}\label{eq:finsize}
\intrL^\ast(\cE) := \sup_{e\in\cE}|e|\sqrt{\frac{\mu(e)}{\nu(e)}} <\infty.
\end{align}
The quantity $\intrL(e):= |e|\sqrt{\frac{\mu(e)}{\nu(e)}}$ is the {\em intrinsic length} of the edge $e\in\cE$. See Section~\ref{ss:IntrMet.01} for further details. 

In at least two special cases, the correspondence between the Kirchhoff Laplacian for $(\cG, \mu, \nu)$ and the discrete Laplacian for the above weights $b$ and $m$ has been known for a quite long time. First of all, in the case of so-called unweighted {\em equilateral metric graphs} (i.e., $\mu=\nu =\id$ on $\cG$ and $|e|=1$ for all edges $e$), \eqref{eq:LaplDiscr} with the weights \eqref{eq:Bbndry},\eqref{eq:Mbndry} turns into the normalized Laplacian \eqref{eq:LaplNorm}. 
Connections between their spectral properties have been established in \cite{ni85}, \cite{vB} for finite metric graphs and then extended in \cite{cat}, \cite{ex97}, \cite{bgp07} to infinite metric graphs, 
 and in fact one can even prove some kind of local unitary equivalence \cite{pan12}. Thus, these results allow to reduce the study of Laplacians on equilateral metric graphs to a widely studied object --- the normalized Laplacian $L_{\rm norm}$, the generator of the simple random walk on $\cG_d$ (see \cite{bar}, \cite{cdv}, \cite{woe}).
The second well-studied case is a slight generalization of the above setting: again, $|e| = 1$ for all edges $e$, however, $\mu=\nu$ on $\cG$ (these are named {\em cable systems} in the work of Varopoulos \cite{var85b}). The corresponding Laplacian $L$ with the coefficients \eqref{eq:Bbndry}, \eqref{eq:Mbndry} is the generator of a discrete time random walk on $\cG_d$ with the probability of jumping from $v$ to $u$ given by 
\begin{align*}
 p (u,v) = \frac{\mu(e_{u,v})}{\sum_{w \sim v} \mu(e_{u,w})}\quad \text{when}\quad u\sim v, 
\end{align*}
and $0$ otherwise. There is a close connection between this random walk and the Brownian motion on the cable system and exactly this link has been exploited several times in the literature (see  \cite{var85b} and some recent works  \cite{baba03}, \cite{fo13}, \cite{fo}).
 
For a given maximal Kirchhoff Laplacian $\bH$ in $L^2(\cG;\mu)$, let us denote the corresponding maximal Laplacian with the weights \eqref{eq:Bbndry}, \eqref{eq:Mbndry} by $\rh = \rh(\cG,\mu,\nu)$. Assuming that the underlying model of $(\cG,\mu,\nu)$ has finite intrinsic size~\eqref{eq:finsize}, it turns out that $\bH$ and $ \rh(\cG,\mu,\nu)$ share many basic properties:

\begin{itemize}
\item[] {\bf Spectral Properties:}
\begin{itemize}
\item {\em Self-adjoint uniqueness}, see \cite[\S~4]{EKMN}, \cite[Chap.~3]{kn21}.
\item {\em Positive spectral gap}, see \cite[\S~4]{EKMN}, \cite{kn19}, \cite[Chap.~3]{kn21}.
\item {\em Ultracontractivity estimates}, see \cite{roso10}, \cite[\S~5.2]{EKMN}, \cite[Chap.~4.8]{kn21}
\end{itemize}
\item[] {\bf Parabolic Properties:}
\begin{itemize}
\item {\em Markovian uniqueness}, see \cite[Chap.~4.4]{kn21}.
\item {\em Recurrence/transience}, see \cite[Chap.~4]{hae}, \cite[Chap.~4.5]{kn21}.
\item {\em Stochastic completeness}, see \cite{fo13}, \cite{hua}, \cite{hks}, \cite{hush14}, \cite[Chap.~4.6]{kn21}.
\end{itemize}
\end{itemize}

The above lists are by no means complete and we refer to the recent monograph~\cite{kn21} for further details, results, and literature. 
 
\begin{remark} 
In fact, the idea to relate the properties of $\Delta$ and $L$ by taking into account the relationship between their kernels has its roots in the fundamental works of M.G.~Krein, M.I.~Vishik and M.Sh.~Birman in the 1950s. Indeed, it turns out that $L$ serves as a ``boundary operator" for $\Delta$ and exactly this fact allows to connect basic spectral properties of these two operators. However, in order to make all that precise one needs to use the machinery of boundary triplets and the corresponding Weyl functions, a modern language of extension theory of symmetric operators in Hilbert spaces, which can be seen as far-reaching development of the Birman--Krein--Vishik theory (see \cite{DM91}, \cite{DM17}, \cite{schm}). First applications of this approach to finite and infinite metric graphs can be traced back to the 2000s (see, e.g., \cite{bgp07}, \cite{ekkst08}, \cite{post}). One of its advantages is the fact that the boundary triplets approach allows to treat metric graphs avoiding the restrictive assumptions on the edge lengths \cite{EKMN}, \cite{KM10}. 
\end{remark}

\section{Cable systems for graph Laplacians}\label{sec:Bndry}

The above considerations naturally lead to the following question: {\em which graph Laplacians may arise as ``boundary operators" for a Kirchhoff Laplacian on a weighted metric graph?} 
Let us be more precise. Suppose a vertex set $\cV$ is given. Each graph Laplacian \eqref{eq:LaplDiscr} is determined by the vertex weight $m\colon \cV\to (0,\infty)$ and the edge weight function $b\colon \cV\times\cV\to [0,\infty)$ having the properties (i)--(iv) of Section~\ref{ss:III.01}. 
 With each such $b$ we can associate a locally finite simple graph $\cG_b = (\cV,\cE_b)$ as described in Remark~\ref{rem:simplevsmult}.

\begin{definition}\label{def:Cable}
A \emph{cable system}\label{not:cable} for a graph $b$ over $(\cV,m)$ is a model 
of a weighted metric graph $(\cG, \mu, \nu)$ having $\cV$ as its vertex set and such that the functions defined by~\eqref{eq:Mbndry} and~\eqref{eq:Bbndry} coincide with $m$ and, respectively, $b$. 
If in addition the underlying graph $(\cV, \cE)$ of the model coincides with $\cG_b = (\cV,\cE_b)$, then the cable system is called {\em minimal}.\label{not:cablemin}
\end{definition}

Thus, the problem stated at the very beginning now can be formulated as follows: {\em Which locally finite graphs $(\cV,m;b)$ have a (minimal) cable system?} It turns out that the existence of a minimal cable system is a nontrivial issue already in the case of a path graph (see \cite[Chap.~5.3]{kn21}). Let us also present the following example.

\begin{example}[Cable systems for $L_{\rm comb}$]\label{ex:CSforComb}
Consider the combinatorial Laplacian $L_{\rm comb}$ on a simple, connected, locally finite graph $\cG_d$, that is, $m\equiv \id$ on $\cV$, $b(u,v) = 1$ exactly when $u\sim v$ and $u\neq v$, and $b(u,v) = 0$ otherwise. It turns out that\footnote{\url{https://mathoverflow.net/questions/59117/}: {\em Assigning positive edge weights to a graph so that the weight incident to each vertex is 1},  (2011).} in this case $(\cV,m;b)$ admits a minimal cable system if and only if for each $e\in\cE$ there is a disjoint cycle cover of $\cG_d$ containing $e$ in one of its cycles.
\end{example}

However, we stress that a general cable system may have loops and multiple edges and thus the simplicity assumption on the model of $(\cG,\mu,\nu)$ (that is, the minimality of a cable system for $(\cV,m;b)$) might be too restrictive. 
Moreover, the underlying combinatorial graph $(\cV, \cE)$ of a cable system for $b$ can always be obtained from the simple graph $\cG_b = (\cV, \cE_b)$ by adding loops and multiple edges. The next result was proved in \cite{fo13} (see also \cite[Chap.~6.3]{kn21}).

\begin{theorem} \label{th:loops}
Every locally finite graph $(\cV, m; b)$ has a cable system.
\end{theorem}

After establishing existence of cable systems, the next natural question is their uniqueness. In fact, every locally finite graph $b$ over $(\cV, m)$ has a large number of cable systems. In particular, the construction in~\cite[Rem.~2, p.~2107]{fo13} is a special case of a general construction using different metrizations of discrete graphs. These connections will be discussed in the next section.

\section{Intrinsic metrics on graphs}\label{sec:IntrMetr}

\subsection{Intrinsic metrics on metric graphs}\label{ss:IntrMet.01}

We define the intrinsic metric $\varrho$ of a weighted metric graph $(\cG, \mu, \nu)$ as the intrinsic metric of its Dirichlet Laplacian $\bH_D$ (in particular, note that $\gQ_D$ is a strongly local, regular Dirichlet form). By \cite[eq.~(1.3)]{stu} (see also \cite[Theorem~6.1]{flw14}), $\varrho_{\rm intr}$ is given by
\[
	\varrho_{\rm intr}(x,y) = \sup \big \{ f(x) - f(y)\, | \, f \in \wh{\cD}_{\loc} \big \}, \qquad x, y \in \cG, 
\]
where the function space $\wh{\cD}_{\loc}$ is defined as
\[
\wh{\cD}_{\loc} = \big\{ f \in H^1_{\loc}(\cG)\, \big | \ \nu(x) |\nabla f (x) |^2 \le \mu(x)\ \ \text{for a.e.}\ x\in\cG \big\}.
\]
In fact, $\varrho_{\rm intr}$ admits an explicit description: define the {\em intrinsic weight}  
\begin{align}\label{eq:etadef}
\eta = \eta_{\mu,\nu} := \sqrt{ \frac{\mu }{ \nu}}\quad \text{on}\ \ \cG.
\end{align}
This weight gives rise to a new measure on $\cG$ whose density w.r.t. the Lebesgue measure is exactly $\eta$ (we abuse the notation and denote with $\eta$ both the edge weight and the corresponding measure).

Recall that a path $\cP$ in $\cG$ is a continuous and piecewise injective map $\cP \colon I\to \cG$ defined on an interval $I\subseteq \R$. If $\cI = [a,b]$ is compact, we call $\cP$ a path with starting point $x:=\cP(a)$ and endpoint $y: =\cP(b)$, and its \emph{(intrinsic) length}  is defined as
\begin{align} \label{eq:LengthMetricPath}
	|\cP|_{\eta} := \sum_j  \int_{\cP((t_j, t_{j+1}))} \eta(\rD x), 
\end{align}
where $a= t_0 < \dots < t_n =b$ is any partition of $\cI = [a,b]$ such that $\cP$ is injective on each interval $(t_j, t_{j+1})$
 (clearly, $|\cP|_\eta$ is well-defined). 
 
 \begin{lemma}\label{lem:IntrMetWMG}
The metric $\varrho_\eta$ defined by 
\begin{equation} \label{eq:compute_metric}
\varrho_\eta(x,y) :=\inf_\cP |\cP|_\eta, \qquad x, y \in \cG,
\end{equation} 
where the infimum is taken over all paths $\cP$ from $x$ to $y$, coincides with the intrinsic metric on $(\cG, \mu, \nu)$ (w.r.t. $\gQ_D$), that is, $\varrho_{\rm intr} = \varrho_\eta$.
\end{lemma}

The proof is straightforward and can be found in, e.g., \cite[Prop.~2.21]{hae} (see also~\cite[Lemma~4.3]{kmn21}). Notice that in the case $\mu=\nu$, $\eta$ coincides with the Lebesgue measure and hence $\varrho_\eta$ is nothing but the length metric $\varrho_0$ on $\cG$ (see Section~\ref{ss:II.02}).  

\begin{remark}\label{rem:intrEdgeLength}
If a path $\cP_e$ consists of a single edge $e\in \cE$, then
\[
|\cP_e|_\eta = \int_e \eta(\rD x) = |e|\sqrt{ \frac{\mu(e) }{ \nu(e)}} = \eta(e),
\]
which connects $\varrho_{\rm intr} = \varrho_\eta$ on $(\cG,\mu,\nu)$ with the intrinsic edge length (see \eqref{eq:finsize}).
\end{remark}

\subsection{Intrinsic metrics on discrete graphs}\label{ss:VI.3.02}
The idea to use different metrics on graphs can be traced back at least to \cite{dav93b} and versions of metrics adapted to weighted discrete graphs have appeared independently in several works, see, e.g., \cite{fo13}, \cite{fo}, \cite{ghm}, \cite{maue11}. In our exposition we follow \cite{flw14}, \cite{kel15}.

For a connected graph $b$ over $(\cV, m)$, a symmetric function $p\colon \cV \times \cV \to [0,\infty)$ such that $p(u,v)>0$ exactly when $b(u,v) >0$ is called a \emph{weight function} for $(\cV,m;b)$. Every weight function $p$ generates a {\em path metric} $\varrho_p$ on $\cV$ w.r.t. $b$ via
\begin{align}\label{eq:pathmetric}
\varrho_p(u,v) := 
\inf_{\cP= (v_0,\dots,v_n)\colon u=v_0,\ v=v_n}\sum_{k} p(v_{k-1},v_k). 
\end{align}
Here the infimum is taken over all paths in $b$ connecting $u$ and $v$, that is, all sequences $\cP = (v_0, \dots, v_n)$ such that $v_0 = u$, $v_n = v$ and $b(v_{k-1}, v_{k}) > 0$ for all $k$. Since we assume $b$ to be locally finite, $\varrho_p(u,v) > 0$ whenever $u \neq v$.

\begin{example}[Combinatorial distance]\label{ex:pathmetrics}
Let $p\colon \cV\times\cV \to \{0,1\}$ be given by 
\begin{align}
p(u,v) = \begin{cases} 1, & b(u,v) \neq 0,\\ 0, & b(u,v) = 0.\end{cases}
\end{align}
Then the corresponding metric $\varrho_p$ is nothing but the combinatorial distance $\varrho_{\rm comb}$ (a.k.a. the {\em word metric} in the context of Cayley graphs) on a graph $b$ over $\cV$.
\end{example}

\begin{definition}[\cite{flw14}]\label{def:IntrMetrGr}
A metric $\varrho$ on $\cV$ is called {\em intrinsic} w.r.t. $(\cV,m;b)$ if 
\begin{align}\label{eq:intrinsicdef}
 \sum_{u\in\cV}b (u,v)\varrho(u,v)^2\le m (v)
\end{align}
holds for all  $v\in\cV$. Similarly, a weight function $p\colon \cV \times \cV \to [0,\infty)$ is called an \emph{intrinsic weight} for $(\cV,m;b)$ if 
\[
	 \sum_{u\in\cV}b (u,v) p(u,v)^2\le m (v),\qquad v\in\cV.
\]
If $p$ is an intrinsic weight, then the path metric $\varrho_p$ is called {\em strongly intrinsic}. 
\end{definition}

For any given locally finite graph $(\cV,m;b)$ an intrinsic metric always exists (see \cite[Example~2.1]{hkmw13}, \cite{kel15} and also \cite{dVTHT}). 

\begin{remark}\label{rem:combIntrMetr}
It is straightforward to check that the combinatorial distance $\varrho_{\rm comb}$ is not intrinsic for the combinatorial Laplacian $L_{\rm comb}$ ($m\equiv \id$ on $\cV$ in this case). On the other hand,  $\varrho_{\rm comb}$ is equivalent to an intrinsic path metric if and only if $\deg$ is bounded on $\cV$, that is, the corresponding graph has bounded geometry. If $\sup_{\cV}\deg(v) = \infty$, then $L_{\rm comb}$ is unbounded in $\ell^2(\cV)$ and it turned out that $\varrho_{\rm comb}$ is not a suitable metric on $\cV$ to study the properties of $L_{\rm comb}$ (in particular, this has led to certain controversies in the past, see \cite{klw}, \cite{woj11}).
\end{remark}

\subsection{Connections between discrete and continuous}\label{ss:VI.3.03}

Consider a weigh\-ted metric graph $(\cG, \mu, \nu)$ and its intrinsic metric $\varrho_\eta$. With each model of $(\cG, \mu, \nu)$ we can associate the vertex set $\cV$ together with the vertex weight $m \colon \cV \to (0, \infty)$ and the graph $b\colon \cV \times \cV \to [0, \infty)$, see \eqref{eq:Mbndry},\eqref{eq:Bbndry}. The next result shows that the intrinsic metric $\varrho_\eta$ of $(\cG, \mu, \nu)$ gives rise to a particular intrinsic metric for $(\cV,m;b)$.

\begin{lemma}[\cite{kn21}] \label{lem:intrinsic_metrics}
Fix a model of $(\cG, \mu, \nu)$ having finite intrinsic size and define the metric $\varrho_\cV$ on $\cV$ as a restriction of $\varrho_\eta$ onto $\cV \times \cV$,
\begin{align}\label{eq:varrho_V}
\varrho_\cV(u,v) := \varrho_\eta(u,v),\qquad (u,v)\in \cV\times\cV.
\end{align}
Then $\varrho_\cV$ is an intrinsic metric for $(\cV,m;b)$. 
Moreover, $(\cG, \varrho_\eta)$ is complete as a metric space exactly when $(\cV, \varrho_\cV)$ is complete.
\end{lemma}

Let us mention that Lemma~\ref{lem:intrinsic_metrics} also has an interpretation in terms of \emph{quasi-isometries} (see, e.g., \cite[Def.~1.12]{bar}, \cite[Sec.~1.3]{nowak}, \cite{roe}).

\begin{definition}\label{def:RoughIso}
A map $\phi \colon X_1 \to X_2$ between metric spaces $(X_1, \varrho_1)$ and $(X_2, \varrho_2)$ is called \emph{a quasi-isometry}\label{not:roughiso} if there are constants $a,R>0$ and $b\ge 0$ such that
\begin{equation} \label{eq:rough_isometry_1}
a^{-1} (\varrho_1(x,y) - b) \le \varrho_2 (\phi(x) ,\phi(y)) \le a (\varrho_1(x,y) + b), 
\end{equation}
for all $x, y \in X_1$ and, moreover, 
\begin{equation} \label{eq:rough_isometry_2}
\bigcup_{x \in X_1} B_R (\phi(x);\varrho_2) = X_2.
\end{equation}
Here and below $B_R (x;\varrho) = \{y\in X\,|\, \varrho(x,y)<R\}$ is a ball in a metric space $(X,\varrho)$.
\end{definition}

It turns out that the map $\imath_\cV$ defined in Section~\ref{sec:connection} is closely related with a quasi-isometry between weighted graphs and metric graphs:

\begin{lemma}\label{lem:RoughIso}
Assume the conditions of Lemma~\ref{lem:intrinsic_metrics}. Then the map
\begin{align}
\varphi \colon \cV \to \cG, \qquad \varphi(v) = v
\end{align}
defines a quasi-isometry between the metric spaces $(\cG, \varrho_\eta)$ and $(\cV, \varrho_\cV)$. 
\end{lemma}

\begin{proof}
The proof is a straightforward check of \eqref{eq:rough_isometry_1} and \eqref{eq:rough_isometry_2} for the map $\phi$ with $a=1$, $b=0$ and 
$R =\eta^\ast(\cE)$ (notice that the finite intrinsic size \eqref{eq:finsize} is necessary for the net property \eqref{eq:rough_isometry_2} to hold).
\end{proof}

\begin{remark}\label{rem:RoughIso}
The notion of quasi-isometries was introduced in the works of M.~Gromov and M.~Kanai in the 1980s. It is well-known in context with Riemannian manifolds and (combinatorial) graphs that roughly isometric spaces share many important properties: e.g., geometric properties (such as volume growth and isoperimetric inequalities), parabolicity/transience, Liouville-type theorems for harmonic functions of finite energy and many more. However, we stress that most of these connections also require additional (rather restrictive) conditions on the local geometry of the spaces. 

Some of our conclusions are reminiscent of this notion,
 but in fact our results go beyond this framework. For instance, the strong/weak Liouville property (i.e., all positive/bounded harmonic functions are constant) is not preserved under bi-Lipschitz maps in general \cite{lyons}. However, the equivalence holds true for our setting (see \cite[Lemma~6.46]{kn21}). In addition, we stress that we do not require any additional local conditions (e.g., bounded geometry). On the other hand, our results connect only two particular roughly isometric spaces $(\cG, \varrho_\eta)$ and $(\cV, \varrho_\cV)$ and not the whole equivalence class of roughly isometric weighted graphs or weighted metric graphs.
\end{remark}

By Lemma~\ref{lem:intrinsic_metrics}, each cable system having finite intrinsic size\footnote{Since by definition a cable system is a model of a weighted metric graph, the notion of intrinsic size  immediately extends to cable systems.} gives rise to an intrinsic metric $\varrho_\cV$ for $(\cV,m;b)$ using a simple restriction to vertices. It is natural to ask which intrinsic metrics on graphs can be obtained as restrictions of intrinsic metrics on weighted metric graphs. Due to the lack of space we omit the description of these results, which roughly speaking  state that {\em to construct an intrinsic metric on a graph $b$ over $(\cV,m)$ is almost equivalent to constructing a cable system}. Let us state only one result here (see Lemma~6.27 and Theorem~6.30 in \cite{kn21}). 

\begin{theorem}[\cite{kn21}]\label{th:metricDvsC}
Let $b$ be a locally finite, connected graph over $(\cV, m)$ equipped with a strongly intrinsic path metric $\varrho$. Assume also that 
$\varrho$ has finite jump size, 
\begin{align*}
s(\varrho) = \sup\{ \varrho(u,v)\,|\, u,v\in\cV, b(u,v)>0\}  < \infty.
\end{align*} 
Then there exists a weighted metric graph $(\cG,\mu,\nu)$ together with a model such that \eqref{eq:finsize} is satisfied, $m$ and $b$ have the form \eqref{eq:Mbndry} and \eqref{eq:Bbndry}, respectively, and, moreover, $\varrho$ coincides with the induced path metric $\varrho_\cV = \varrho_\eta|_{\cV\times\cV}$.
\end{theorem}

\begin{remark}
It is hard to overestimate the role of intrinsic metrics in the progress achieved for weighted graph Laplacians during the last decade. Surprisingly, the above described procedure to construct an intrinsic metric for $(\cV,m;b)$ in fact provides a way to obtain all finite jump size intrinsic path metrics on $(\cV,m;b)$. Moreover, upon normalization assumptions on cable systems (e.g., restricting to weighted metric graphs with equal weights, i.e., $\mu=\nu$, and also assuming no multiple edges and that all loops have the same length $1$) the correspondence in Theorem~\ref{th:metricDvsC} becomes in a certain sense bijective  (see \cite[Theorem~6.34]{kn21}). \\
Let us mention that some versions of this result have been used earlier in \cite{fo13}, \cite{hua}.
\end{remark} 

\section{Applications}\label{sec:Applic}

Our main goal in this final section  is to demonstrate the established connections between graph Laplacians and metric graph Laplacians. 
We will describe some applications to the self-adjointness problem and to the problem of recurrence.
For further results as well as applications (Markovian uniqueness, spectral gap estimates, stochastic completeness etc.) we refer to \cite[Chap.~7--8]{kn21}.
  
\subsection{Self-adjointness}\label{ss:applSA}

The first mathematical problem arising in any quantum mechanical model is {\em self-adjointness} (see, e.g., \cite[Chap.~VIII.11]{RSI}), that is, usually a formal symmetric expression for the Hamiltonian has some natural domain of definition in a given Hilbert space (e.g., pre-minimally defined Laplacians) and then one has to verify that it gives rise to an (essentially) self-adjoint operator. Otherwise\footnote{Of course, one needs to check whether the corresponding symmetric operator has equal deficiency indices, which is always the case for Laplacians or, more generally,  for symmetric operators which are bounded from below or from above.}, there are infinitely many self-adjoint extensions (or restrictions in the maximally defined case) and one has to determine the right one which is the observable. 

There are several ways to introduce the notion of self-adjointness. For the Kirchhoff Laplacian as well as for the graph Laplacian (take into account the locally finite assumption) the self-adjointness means that the minimal Laplacian coincides with the maximal Laplacian in the corresponding $L^2$ space. 
On the other hand, considering the associated Schr\"odinger or wave equation, the self-adjointness actually means its $L^2$-solvability (see, e.g., \cite[\S 1.1]{shu92}). Perhaps, the most convenient way for us would be to define the self-adjointness via solutions to the Helmholz equation
\begin{align}\label{eq:helmholz}
\Delta u = \lambda u,\qquad \lambda\in \R.
\end{align} 
Since $\Delta$ is non-positive, the maximally defined operator is self-adjoint if and only if for some (and hence for all) $\lambda>0$ equation \eqref{eq:helmholz} admits a unique solution $u\in L^2(\cG;\mu)$, which is clearly identically zero in this case  (see, e.g., \cite[Theorem~X.26]{RSII}).
Recalling that, in the context of both manifolds and graphs, functions satisfying \eqref{eq:helmholz} are called $\lambda$-harmonic, the self-adjoint uniqueness can be seen as some kind of a Liouville-type property of $\cG$\footnote{Under the positivity of the spectral gap one can in fact replace $\lambda>0$ by $\lambda=0$ and hence in this case one is led to harmonic functions.} and this indicates its close connections with the geometry of the underlying metric space. We begin with the following result, which is widely known in the context of Riemannian manifolds.\footnote{M.P.~Gaffney~\cite{gaf55} noticed the importance of completeness of the manifold in question and the essential self-adjointness in this case was established later~\cite{roe} (see also~\cite{che}, \cite{str}).}.

\begin{theorem}\label{lem:Gaffney} 
Let $(\cG,\mu,\nu)$ be a weighted metric graph and let $\varrho_\eta$ be the corresponding intrinsic metric. 
If $(\cG,\varrho_\eta)$ is complete as a metric space, then the Kirchhoff Laplacian $\bH$ is self-adjoint.
\end{theorem}

In the context of metric graphs, the above result seems to be a folklore, however, it is not an easy task to find its proof in the literature. In fact, the above considerations enable us to provide a rather short one.

\begin{proof}
Assume that $\bH$ is not self-adjoint. Since the minimal Kirchhoff Laplacian $\bH^0 = \bH^\ast$ is nonnegative, this means that $\ker(\bH + \rI) \neq \{0\}$, that is, there is $0\neq f\in\dom(\bH)$ such that $\Delta f = f$ (see \cite[Theorem~X.26]{RSII}). Moreover, we can choose such an $f$ real-valued and hence $|f|$ is subharmonic. Moreover, $|f|\in L^2(\cG;\mu)$ since $f\in \dom(\bH)$. On the other hand, if $(\cG,\varrho_\eta)$ is complete as a metric space, then applying Yau's $L^p$-Liouville theorem \cite[Cor.~1(a)]{stu}, we conclude that $f\equiv 0$. This contradiction completes the proof.
\end{proof}

\begin{remark}
A few remarks are in order.
\begin{itemize}
\item[(i)]
Simple examples (e.g., $\cG$ is a path graph) show that the completeness w.r.t. $\varrho_\eta$ is not necessary. 
\item[(ii)]
By the Hopf--Rinow theorem (a metric graph $\cG$ equipped with $\varrho_\eta$ is a length space) completeness of $(\cG,\varrho_\eta)$ is equivalent to bounded compactness (compactness of distance balls), as well as to geodesic completeness.
\end{itemize}
\end{remark}

As an immediate corollary of Theorem~\ref{lem:Gaffney} and the above discussed connections, we obtain a version of the Gaffney theorem for graph Laplacians. 

\begin{theorem}[\cite{hkmw13}]\label{cor:GaffneyDiscr} 
Let $b$ be a locally finite graph over $(\cV,m)$ and let $\varrho$ be an intrinsic metric which generates the discrete topology on $\cV$. If $(\cV,\varrho)$ is complete as a metric space, then $\rh^0$ is self-adjoint and $\rh^0 = \rh$.
\end{theorem}

\begin{proof}
Let us only sketch the proof (missing details can be found in~\cite[Chap.~7.1]{kn21}). By Theorem~\ref{th:metricDvsC}, there is a cable system for $(\cV,m;b)$. Moreover, the corresponding Kirchhoff Laplacian $\bH$ is self-adjoint if and only if so is $\rh$ (see \cite[\S~4]{EKMN}, \cite[Theorem~3.1(i)]{kn21}). Taking into account Lemma~\ref{lem:intrinsic_metrics} and applying Theorem~\ref{lem:Gaffney}, we complete the proof.  
\end{proof}

\begin{remark}
A few remarks are in order.
\begin{itemize}
\item[(i)]
Theorem~\ref{cor:GaffneyDiscr} was first established in \cite[Theorem~2]{hkmw13}. 
\item[(ii)]
Both Theorem~\ref{lem:Gaffney} and Theorem~\ref{cor:GaffneyDiscr}) are not optimal. For instance, Theorem~\ref{cor:GaffneyDiscr} does not imply the self-adjointness of the combinatorial Laplacian $L_{\rm comb}$ when it is unbounded (see \cite{jor}, \cite[Theorem~6]{kl12}). However, Theorems~\ref{lem:Gaffney} and~\ref{cor:GaffneyDiscr} enjoy a certain stability property under additive perturbations, which preserve semiboundedness (\cite[Theorem~2.16]{gks15}, \cite{kmn21}).
\item[(iii)] We refer for further results and details to \cite{kmn21}, \cite[Chap.~7.1]{kn21}, and \cite{schm20}.
\end{itemize}
\end{remark}

\subsection{Recurrence and transience}\label{ss:applRecur}

Recurrence of a random walk/Brownian motion means that a particle returns to its initial position infinitely often (see, e.g., \cite{fuk10} for a detailed definition). In fact, recurrence appears (quite often under different names) in many different areas of mathematics and mathematical physics and enjoys deep connections to various important problems (e.g., the type problem for simply connected Riemann surfaces). 

The famous theorem of G.~P\'olya states that the simple random walk on $\Z^d$ is recurrent if and only if either $d=1$ or $d=2$. 
Intuitively one may explain recurrence of a random walk/Brownian motion as insufficiency of volume in the state space (volume of a ball of radius $R$ in $\Z^d$ or $\R^d$ grows faster as $R\to \infty$ for larger $d$). The qualitative form of this heuristic statement in the manifold context has a venerable history (we refer to the excellent exposition of A.~Grigor'yan~\cite{gri99} for further details) and in the case of complete Riemannian manifolds  it was proved in the 1980s independently by L.~Karp, N.Th.~Varopoulos, and A.~Grigor'yan (see \cite[Theorem~7.3]{gri99}) that 
\[
\int^\infty \frac{r\, dr}{{\vol}(B_r(x))} = \infty
\]
guarantees recurrence. Moreover, this condition is close to be necessary. This result was extended to strongly local Dirichlet forms by K.-T.~Sturm~\cite{stu} and hence it also holds in the setting of weighted metric graphs. Again, using the obtained connections between metric graph and weighted graph Laplacians, we can proceed as in the previous subsection and establish the corresponding volume growth test for weighted graph Laplacians, which was originally obtained by B.~Hua and M.~Keller~\cite{huke14}. Due to the lack of space we only refer to \cite[Chap.~7.4]{kn21} for further details. 

We would like to finish this article by reflecting on another interesting topic. Perhaps, the most studied class of graphs are Cayley graphs of finitely generated groups (Example~\ref{ex:Cayley}). Random walks on groups is a classical and very rich subject (the literature is enormous  and we only refer to the classic text \cite{woe}). 
 Recall that a group $\mG$ is called {\em recurrent} if the simple random walk on its Cayley graph $\cC(\mG,S)$ is recurrent for some (and hence for all) $S$. The classification of recurrent groups was accomplished in the 1980s by proving the famous {\em Kesten conjecture}. It is a combination of two seminal theorems -- relationship between decay of return probabilities and growth in groups 
established by N.Th.~Varopoulos and the characterization of groups of polynomial growth by M.~Gromov (see, e.g., \cite[Chap.~VI.6]{var}, \cite[Theorem~3.24]{woe}). 

\begin{theorem}[N.Th.~Varopoulos]\label{th:VaropRecur}
$\mG$ is recurrent if and only if $\mG$ contains a finite index subgroup isomorphic either to $\Z$ or to $\Z^2$. 
\end{theorem}

It turns out that the problem of recurrence on weighted metric graphs can be reduced to the study of recurrence of random walks on groups (see \cite[Theorem~7.49]{kn21}). Let $(\cG_C,\mu,\nu)$ be a weighted metric graph where $\cG_C = \cC(\mG,S)$ is a Cayley graph of a finitely generated group $\mG$. Also, let $\bH_D$ be the corresponding Dirichlet Laplacian. Define
\begin{align}\label{eq:RWmuG}
b_\nu(u,v) =  \begin{cases} \frac{\nu(e_{u,v})}{|e_{u,v}|}, & u^{-1}v \in S,\\ \ \ 0, & u^{-1}v\notin S,\end{cases}\qquad u,v\in \mG.
\end{align} 

\begin{theorem}\label{cor:RWonGroup}
 The heat semigroup $(\E^{-t \bH_D})_{t>0}$ is recurrent if and only if the discrete time random walk on $\mG$ with transition probabilities 
\begin{align}\label{eq:PnuRW}
p_\nu(u,v) = P(X_{n+1} = v\,|\, X_{n} = u) = \frac{b_\nu(u,v)}{\sum_{g\in S} b_\nu (u,ug)}, \qquad u,v \in \mG, 
\end{align}
is recurrent.
\end{theorem}

Combining this result with Theorem~\ref{th:VaropRecur}, we arrive at the following result:

\begin{corollary}\label{th:MGrecur}
Assume the conditions of Theorem~\ref{cor:RWonGroup}.
\begin{itemize}
\item[(i)]
If $\mG$  contains a finite index subgroup isomorphic either to $\Z$ or to $\Z^2$ and 
the edge weight $\nu$ satisfies
\begin{align}\label{eq:recur01}
\sup_{e\in\cE}\frac{\nu(e)}{|e|} <\infty,
\end{align}
then the heat semigroup  $(\E^{-t \bH_D})_{t>0}$ is recurrent.
\item[(ii)]
If $\mG$ is transient (i.e., $\mG$ does not contain a finite index subgroup isomorphic either to $\Z$ or to $\Z^2$) and 
the edge weight $\nu$ satisfies
\begin{align}\label{eq:recur02}
\inf_{e\in\cE}\frac{\nu(e)}{|e|} > 0,
\end{align}
then the heat semigroup  $(\E^{-t \bH_D})_{t>0}$ is transient.
\end{itemize}
\end{corollary}

In fact, the above result have numerous consequences and actually can be improved. Let us finish by its applications to ultracontractivity estimates. To simplify the exposition we restrict now to unweighted metric graphs.

\begin{theorem}[\cite{EKMN,kn21}]\label{th:UltraGroup}
Assume the conditions of Theorem~\ref{cor:RWonGroup} and let also $\mu = \nu\equiv \id$.
Suppose that $\mG$ is not recurrent and the edge lengths satisfy
\begin{align}\label{eq:munuUltra1}
\sup_{e\in\cE} |e|  < \infty. 
\end{align}
Then $(\E^{-t\bH_D})_{t>0}$ is ultracontractive and, moreover, 
\begin{itemize}
\item[(i)]
If $\gamma_\mG(n) \approx n^N$
as $n\to \infty$ with some $N\in \Z_{\ge 3}$, then\footnote{Here $\gamma_\mG\colon \Z_{\ge 0} \to \Z_{>0}$ is the growth function defined by $\gamma_\mG(n) = \#\{ g\in \mG\, |\, \varrho_{\rm comb}(g,o) \le n\}$, where $o$ is the identity element of $\mG$.} 
\begin{align}\label{eq:UltraD/2}
\|\E^{-t\bH_D}\|_{1\to \infty} \le C_N t^{-N/2},\qquad t>0.
\end{align}
\item[(ii)]
If $\mG$ is not virtually nilpotent (i.e., $\gamma_\mG$ has superpolynomial growth\footnote{This means that for each $N>0$ there is $c>0$ such that $\gamma_\mG(n) \ge cn^N$ for all large $n$.}), then \eqref{eq:UltraD/2} holds true for all $N>2$.
\end{itemize}
\end{theorem}

\begin{remark}
Notice that applying Theorem~1.2 and Theorem~1.3 from \cite{ls97} to the Dirichlet Laplacian $\bH_D$ and then using Theorem~\ref{th:UltraGroup}, we arrive at the Cwikel--Lieb--Rozenblum estimates for additive perturbation, that is, for Schr\"odinger operators  $-\Delta +V(x)$. It is also well known (see \cite{fls}) that ultracontractivity estimates and Sobolev-type inequalities lead to Lieb--Thirring bounds ($\mathfrak{S}_p$ estimates on the negative spectra), however, we are not going to pursue this goal here.
For further details and historical remarks we refer to \cite[Chap.~8.2]{kn21}.
\end{remark}


\end{document}